\documentclass{amsart}
\usepackage{amsfonts,amssymb}
\usepackage{bm}
\usepackage{graphicx}
\usepackage{hyperref}
\newtheorem{theorem}{Theorem}
\theoremstyle{plain}

\newtheorem{corollary}{Corollary}

\newtheorem{lemma}{Lemma}

\newtheorem{proposition}{Proposition}
\newtheorem{remark}{Remark}

\numberwithin{equation}{section}

\begin{document}
\title[Steiner $(k,l)$-eccentricity]{Some sharp bounds on the average Steiner $(k,l)$-eccentricity for trees}
\author[Cheng ZENG]{Cheng ZENG}
\address{Shandong Technology and Business University}
\email{czeng@sdtbu.edu.cn}
\author[Gengji LI]{Gengji LI}
\address{Shandong Technology and Business University}
\email{lgengji@gmail.com}
\date{}
%\dedicatory{Dedicated to the memory of S. Bach.}
\thanks{Cheng Zeng is the corresponding author.}

\begin{abstract}
The Steiner $(k,l)$-eccentricity of a vertex $l$-set $S$ is the maximum Steiner $k$-distance among all vertex sets containing  $S$. In this paper we introduce some transformations for trees that do not increase the average Steiner $(k,l)$-eccentricity for all $0\leq l\leq k\leq n$. Using these transformations, we obtain some sharp bounds on the average Steiner $(k,l)$-eccentricity for trees with some certain conditions, including given nodes, given diameter, given maximum degree and given leaves, and get the corresponding extremal trees as well.
\end{abstract}

\maketitle

\section{Introduction}
Throughout this paper we consider simple, undirected and connected graphs. Let $G=(V(G), E(G))$ be a graph with vertex set $V(G)$ and edge set $E(G)$. Let $\#A$ be the cardinality of set $A$. We call $\#V(G)$ and $\#E(G)$ the \textit{order} and the \textit{size} of graph $G$, respectively.  We always assume $\#V(G)=n$. We say $A$ a $k$-set if $\#A=k$. Given a $k$-set $S\subseteq V(G)$, the size of minimal spanning tree of $S$ in $G$ is called the \textit{Steiner $k$-distance} of $S$ and we denote it by $d_G(S)$. Especially, if $S$ is a $2$-set $\{u,v\}$, then $d_G(S)=d_G(u,v)$ is the classical distance. 

In the past few decades, the Steiner distance and the Steiner eccentricity have witnessed a surge in research activities, see the surveys \cite{chartrand1989steiner, dankelmann1996average, ali2012upper, mao2017steiner, li2021average, dankelmann2021bounding}. The \textit{Steiner $(k,l)$-eccentricity} of $l$-set $S$ in graph $G$ is given by
\begin{equation*}
	\varepsilon_{k,l}(S;G)=\max\{d_G(S'): S\subseteq S'\subseteq V(G),  \#S'=k\geq\max\{2,l\}\}.
\end{equation*}
In cases  without confusion, we also denote it by $\varepsilon_{k,l}(S)$. In particular, $\varepsilon_2(v):=\varepsilon_{2,1}(v;G)$ and $\varepsilon_k(v):=\varepsilon_{k,1}(v;G)$  are the \textit{eccentricity} and the \textit{Steiner $k$-eccentricity} of $v$ in $G$; $\operatorname{diam}_k(G):=\varepsilon_{k,0}(\emptyset)=\max_{S\subseteq V(G)}\varepsilon_{k,l}(S;G)$ is the \textit{Steiner $k$-diameter} of $G$;  $\operatorname{rad}_{k,l}(G):=\min_{S\subseteq V(G)}\varepsilon_{k,l}(S;G)$ is the \textit{Steiner $(k,l)$-radius} of $G$; $\varepsilon_{k,l}(S;G)=d_G(S)$ if $k=l$.

The \textit{average Steiner $(k,l)$-eccentricity} of $G$ is the arithmetic mean of the Steiner $(k,l)$-eccentricities of all $l$-sets in $V(G)$, that is, 
\begin{equation*}
	\bar{\varepsilon}_{k,l}(G)=\binom{n}{l}^{-1}\sum_{\#S=l, S\subseteq V(G)}\varepsilon_{k,l}(S;G).
\end{equation*}
The \textit{average eccentricity} and the \textit{average Steiner $k$-eccentricity} of $G$ are thus given by $\bar{\varepsilon}(G):=\bar\varepsilon_{2,1}(G)$ and $\bar{\varepsilon}_k(G):=\bar{\varepsilon}_{k,1}(G)$, respectively.

With respect to the definition of the Steiner distance, Li et al. \cite{li2016steiner} gave the concept of the \textit{Steiner $k$-Wiener index}, that is,
\begin{equation*}
	SW_k(G):=\sum_{S\subseteq V(G), \#S=k}d_G(S),~2\leq k\leq n-1.
\end{equation*}
They proved that
\begin{equation*}
	\binom{n-1}{k-1}(n-1)\leq SW_k(T) \leq \binom{n+1}{k+1}(k-1),
\end{equation*}
where the star $K_{1,n-1}$ and the path $P_n$ reach the lower and upper bounds, respectively. For a fixed $n$, since $\bar{\varepsilon}_{k,k}(G)=\binom{n}{k}^{-1}SW_k(G)$, we plainly have 
\begin{equation*}
	\bar\varepsilon_{k,k}(K_{1,n-1})\leq \bar\varepsilon_{k,k}(T)\leq\bar\varepsilon_{k,k}(P_n).
\end{equation*}

Very recently, Li  et al. \cite{ li2021average,li2021steiner} obtained the lower and upper bounds on the average Steiner $k$-eccentricity for trees of order $n$ and for all $k$, that is, 
\begin{equation*}
	\bar\varepsilon_k(K_{1,n-1})\leq \bar\varepsilon_k(T)\leq\bar\varepsilon_k(P_n).
\end{equation*}

In our previous research \cite{li2024average}, we derived 
\begin{equation*}
	\bar\varepsilon_{3,2}(K_{1,n-1})\leq \bar\varepsilon_{3,2}(T)\leq\bar\varepsilon_{3,2}(P_n)
\end{equation*}
by constructing some fascinating transformations. 

An interesting question is whether there are similar conclusions for the general Steiner $(k,l)$-eccentricity. In this paper we provide an affirmative response.  

The paper is organized as follows. In Section \ref{sec2}, we introduce some basic properties and notations for the Steiner $(k,l)$-eccentricity. In Sections \ref{sec3} and \ref{sec4}, we present two crucial transformations on trees, $\sigma$-transformation \cite{li2024average} and $(p,q)$-transformation \cite{ilic2012extremal}, which are monotonic with respect to the average Steiner $(k,l)$-eccentricity. For a more general characterization of the Steiner $(k,l)$-eccentricity, additional conditions have been incorporated into these transformations, please see Lemma \ref{lem1} and Theorem \ref{thm3} for details. In Section \ref{sec5}, we obtain some lower and upper bounds on the average Steiner $(k,l)$-eccentricity of trees with given order, given maximum degree, given number of leaves, given diameter in turn, and give the corresponding extremal trees as well. We also give some sharp bounds on the Steiner Wiener index for trees.
\section{Preliminary and basic properties}\label{sec2}
In this paper we consider the average Steiner $(k,l)$-eccentricity of tree $T$. We list some notations and definitions needed in this paper. The \textit{degree} of vertex $v$ in $G$, denoted by $\deg(v;G)$ ($\deg(v)$ for short), is the number of neighbors of $v$ in $G$. A vertex $v$ is a \textit{branching vertex} if $\deg(v)\geq3$. Let $\Delta(G)$ be the maximum degree of $G$. We call a vertex a \textit{leaf} if its degree is $1$. The set $\ell(T)$ consists of all leaves in $G$. We say an edge of $G$ a \textit{pendant edge} if it possesses a leaf. A path in $G$ is called a \textit{pendant path} if it contains a leaf and a branching vertex as endpoints, and each of its internal vertices has degree $2$.  The \textit{diameter} and the \textit{radius} of graph $G$, denoted by $\operatorname{diam}(G)$ and $\operatorname{rad}(G)$, are $\max_{v\in V(G)}\varepsilon_2(v)$ and $\min_{v\in V(G)}\varepsilon_2(v)$, respectively. A \textit{diametrical path}  of $G$ is a path in $G$ and of length $\operatorname{diam}(G)$.

We call $T_{k,l}(S)$ a \textit{Steiner $(k,l)$-tree} of $l$-set $S$ if there exists a $k$-set $S'$ containing $S$ such that the size of  $T_{k,l}(S)$, as the minimal spanning tree of $S'$, equals $\varepsilon_{k,l}(S;G)$.  Let $EV(T_{k,l}(S))=\{u: u\in S'\backslash S, S'\subseteq T_{k,l}(S)\}$ be the set of all eccentric vertices of $S$ in $T_{k,l}(S)$. For a fixed $S$, all possible Steiner $(k,l)$-trees of $S$ in $T$ form a set $\mathcal T_{k,l}(S)$.
\begin{proposition}\label{prop1}
	The eccentric vertices of $T_{k,l}(S)$ take leaves in $T$ with priority.
\end{proposition}
\begin{proof}
Suppose that an internal vertex $\omega$ is one of the eccentric vertices of $T_{k,l}(S)$ and some leaves remain after we delete $EV(T_{k,l}(S))\cup S$ from $T$. Since $T$ is connected, we may set one of the remaining leaves as new $\omega$. Hence, a larger Steiner $(k,l)$-tree containing $S$ will be obtained, a contradiction.
\end{proof}
\begin{proposition}\label{prop2}
	If $n\geq k+1$ and $\#\ell(T)\leq k-l$, then $\varepsilon_{k,l}(S;T)=n-1$ for all $l$-sets $S$, and hence $\bar{\varepsilon}_{k,l}(T)=n-1$.
\end{proposition}
\begin{proof}
By Proposition \ref{prop1}, $T_{k,l}(S)$ contains all leaves of $T$. So we have $\varepsilon_{k,l}(S;T)=n-1$ and then $\bar{\varepsilon}_{k,l}(T)=n-1$.
\end{proof}
\begin{remark}
	 One can easily see that $\#\mathcal T_{k,l}(S)=1$ if $\#\ell(T)\leq k-l$ or $k=l$; $\#\mathcal T_{k,l}(S)\geq1$ if $k>l$ and $\#\ell(T)> k-l$. The case $k=l$ follows from the acyclic property of a tree, or equivalently, the uniqueness of the minimal spanning tree of a vertex set within a tree. The case $\#\ell(T)\leq k-l$ comes from Proposition \ref{prop2}. The last case holds since we may have more choices of leaves as eccentric vertices.
\end{remark}

Without loss of generality, we always assume that $k-l<\#\ell(T)$ and $n\geq k+1$ in the rest parts of our paper.

\begin{proposition}
Let $P_n$ be the path of order $n>k$. We have
\begin{enumerate}
	\item $\bar{\varepsilon}_{k,l}(P_n)=
		\frac{(k-1)(n+1)}{k+1}$, if $k-l=0$;
   \item If $k-l=1$ and $k\geq3$, then \begin{align*}\begin{split}
   	\bar{\varepsilon}_{k,l}(P_n)=\binom{n}{k-1}^{-1}\sum_{d=k-2}^{n-1}\binom{d-1}{k-3} \frac{1}{2}\left(n^2-n-d^2+d+2h\left(n-d-h\right)\right),\end{split}\end{align*} 
   where $h=\left\lceil\frac{n-1-d}{2}\right\rceil$; 
   \item $\bar{\varepsilon}_{2,1}(P_n)=\left\lfloor\frac{3 n-2}{4}\right\rfloor$;
	\item $\bar{\varepsilon}_{k,l}(P_n)=n-1$, if $k-l\geq2$.
\end{enumerate}

We also have$$
\bar{\varepsilon}_{k,l}(K_{1,n-1})=k-\frac{l}{n},$$
where $l\leq k<n$, and
$$\bar{\varepsilon}_{n,l}(K_{1,n-1})=n-1.$$
\end{proposition}

\section{$\sigma$-transformation}\label{sec3}
Let $\varepsilon_{X}(v)=\max\{d_G(u,v): u\in X\}$, where $X$ is a connected subgraph of $G$.  Suppose that $T$ can be described as shown in Figure \ref{Fig1}.  In detail, $X$, $Y$, $Z$ are subtrees of $T$ and connected by path $P=P(v_0,v_s)$ of length $s$. We also require that $\varepsilon_{Y}(v_0)\geq \max\{\varepsilon_{X}(v_s), \varepsilon_{Z}(v_s)\}$. We call this form the \textit{$\sigma$-form} of tree $T$. Notice that $X$ can be a singleton graph.

To provide a unified characterization of $\bar{\varepsilon}_{k,l}(T)$, especially for $k=l$, we need to give an extra restriction on $\sigma$-form of tree $T$, that is, $\#V(Y)\geq \#V(X)$. Under the restriction, we can define a transformation on tree $T$, named the \textit{$\sigma$-transformation}, by moving the subtree $Z$ from $v_s$ to $v_0$. Set $T'=\sigma(T)$. Clearly, the inverse of $\sigma$-transformation on $T'$ is defined by moving $Z$ from $v_0$ to $v_s$. See Figure \ref{Fig1}.  We call the transformed tree $T'$  the \textit{$\sigma^{-1}$-form} of $T'$, as shown on the right side of Figure \ref{Fig1}. 

\begin{figure}[h!]
	\includegraphics[width=1\textwidth]{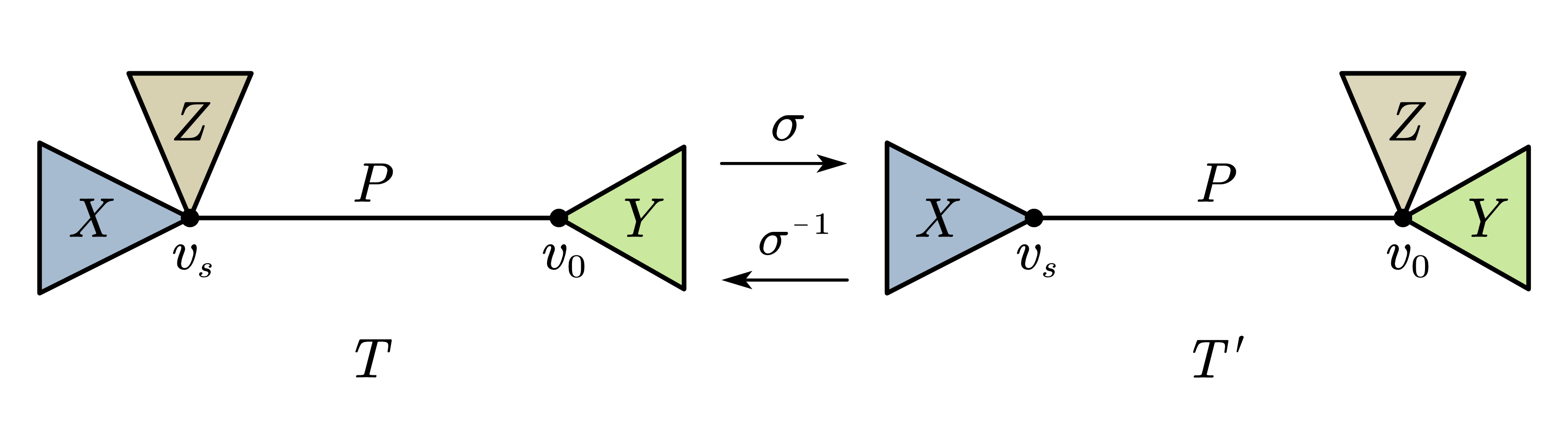}
	\caption{$\sigma$-transformation of $T$}\label{Fig1}
\end{figure}

The following lemma shows the existences of $\sigma$-transformation and its inverse.
\begin{lemma}\label{lem1}
Assume that all trees are of order $n\geq k+1$. If $T\ncong K_{1,n-1}$, then there exists a  $\sigma$-transformation on $T$.  If $T'\ncong P_n$, then there exists a $\sigma^{-1}$-transformation on $T'$.
\end{lemma}
\begin{proof}
	There exists a branching vertex on $T$ such that it has a pendant path $P'$ of length $\leq \operatorname{rad}(T)$. Let $X$ be an internal vertex of $P'$. Hence, we can  obtain a $\sigma$-transformation on $T$.
	
	There exists a branching vertex on $T'$ such that it has a pendant path $P''$ of length $\leq \operatorname{rad}(T)$. Let $P''=X\cup P$. The choices of $Y$ and $Z$ are now obvious. Hence, the existence of $\sigma^{-1}$-transformations on $T'$ is confirmed.
\end{proof}
For convenience, assume that $T'_{k,l}(S)$, the subtree of $T'$, is obtained from $T_{k,l}(S)$ by applying $\sigma$-transformation or $(p,q)$-transformation once. 
\begin{lemma}\label{lem2}
For the $\sigma$-form of tree $T$, if there exists a subtree  $T_{k,l}(S)\in\mathcal T_{k,l}(S)$ such that the path $P$ is contained in $T_{k,l}(S)$, then we have $\varepsilon_{k,l}(S;T)=\varepsilon_{k,l}(S;T')$.
\end{lemma}
\begin{proof}
It is clear that the size of $T_{k,l}(S)$ is unchanged after $\sigma$-transformation. So, 	
\begin{equation}\label{eq3.1}
\varepsilon_{k,l}(S;T)\leq\varepsilon_{k,l}(S;T').
\end{equation}
Now suppose that every $T'_{k,l}(S)$ does not contain path $P$. We then get $T'_{k,l}(S)\subseteq Z\cup Y\subseteq T'$ and hence $\varepsilon_{k,l}(S;T)\geq \varepsilon_{k,l}(S;T')+d$, a contradiction to \eqref{eq3.1}. So there exists a subtree $T'_{k,l}(S)\subseteq T'$ such that $P\subseteq T'_{k,l}(S)$. Similar to \eqref{eq3.1}, we have $$\varepsilon_{k,l}(S;T)\geq\varepsilon_{k,l}(S;T').$$
The proof is now complete.
\end{proof}
Let $(A)$ be the the Iverson convention, where $(A)=1$ if the statement $A$ is true, and $(A)=0$ if $A$ is false.  Let $\bm{x}(H)\in\{0,1\}^4$  be a  vector, where the coordinate components are $(H\cap(V(Y)\backslash\{v_0\})\neq\emptyset)$, $(H\cap(V(X)\backslash\{v_s\})\neq\emptyset)$, $(H\cap(V(Z)\backslash\{v_s\})\neq\emptyset)$, $(H\cap V(P)\neq\emptyset)$ in turn and $H$ is a subset of $V(T)$. We also let $\bm{x}_i$ be the $i$-th component of $\bm{x}$.

We consider two vectors representing the statements of the Steiner tree, that is, $\bm{x}'=\bm{x}(S)$ and $\bm{x}''=\bm{x}(EV(T_{k,l}(S)))$. 
Clearly, $\bm{x}''=(0,0,0,0)$ means that $k=l$ and thus $\varepsilon_{k,l}(S;T)=d_T(S)$ is the Steiner distance of $S$ in $T$.

The following two lemmas are crucial to solve the Steiner $(k,l)$-eccentricity problem.
\begin{lemma}\label{lem3}
For $T\ncong K_{1,n-1}$, $n\geq k+1$ and $k>l$, we have $\bar\varepsilon_{k,l}(T)\geq\bar\varepsilon_{k,l}(T')$, where the equality holds if $\varepsilon_{X}(v_s)=\varepsilon_{Y}(v_0)$.
\end{lemma}
\begin{proof}
We consider the following cases.
\begin{enumerate}
	\item $\bm{x}'=(1,0,0,1)$, $(0,0,0,1)$, $(1,0,0,0)$.
	
	If there exists a tree $T_{k,l}(S)$ such that $\bm{x}''_3=0$, then we have $\varepsilon_{k,l}(S;T)=\varepsilon_{k,l}(S;T')$ by Lemma \ref{lem2}; Otherwise, $\bm{x}''_3=1$ holds for all $T_{k,l}(S)$. We now check $\bm{x}''_1$. When $\bm{x}''_1=1$  for some $T_{k,l}(S)$, we also derive $\varepsilon_{k,l}(S;T)=\varepsilon_{k,l}(S;T')$. When $\bm{x}''_1=0$, we get $\varepsilon_{X}(v_s)\leq\varepsilon_{Z}(v_s)$. Thus,  $\varepsilon_{k,l}(S;T)\geq\varepsilon_{k,l}(S;T')$, where the equation holds if $\varepsilon_{X}(v_s)=\varepsilon_{Y}(v_0)$.
	\item $\bm{x}'=(1,0,1,0)$, $(1,0,1,1)$.
	
	If there exists a tree $T_{k,l}(S)$ such that $\bm{x}''_2=1$, then  $\varepsilon_{k,l}(S;T)=\varepsilon_{k,l}(S;T')$. Otherwise, for all $T_{k,l}(S)$, $\bm{x}''_2=0$. Therefore, $\varepsilon_{k,l}(S;T)\geq\varepsilon_{k,l}(S;T')$. Note that $\varepsilon_{X}(v_s)=\varepsilon_{Y}(v_0)$ implies $\bm{x}''_2=1$.

\item $\bm{x}'=(0,0,1,0)$, $(0,0,1,1)$.

 $\bm{x}''_1=1$ surely holds by the condition of $\sigma$-transformation. Suppose there exists a tree $T_{k,l}(S)$ such that $\bm{x}''_2=1$. We have $\varepsilon_{k,l}(S;T)=\varepsilon_{k,l}(S;T')$. We now consider the subcase: $\forall T_{k,l}(S)$, $\bm{x}''_2=0$. We have 
 \begin{equation*}
 	\varepsilon_{k,l}(S;T)-\varepsilon_{k,l}(S;T')\geq \min\{d, d_T(S,v_s)\}\geq 0,
 \end{equation*}
where the equation holds if $\varepsilon_{X}(v_s)=\varepsilon_{Y}(v_0)$.

\item $\bm{x}''=(0,1,\{0,1\}^2)$, $(1,1,\{0,1\}^2)$.

For these cases, one can easily check that there exists a tree $T_{k,l}(S)$ such that $P\subseteq T_{k,l}(S)$. It follows that $\varepsilon_{k,l}(S;T)=\varepsilon_{k,l}(S;T')$.
\end{enumerate}

Combining the above discussions yields
\begin{equation}\label{eq3.2}
	\sum_{S\subseteq V(T)}\varepsilon_{k,l}(S;T)\geq \sum_{S\subseteq V(T')}\varepsilon_{k,l}(S;T'),
\end{equation}
where the equation holds if $\varepsilon_{X}(v_s)=\varepsilon_{Y}(v_0)$. Taking the arithmetic mean, \eqref{eq3.2} shows that  $\bar\varepsilon_{k,l}(T)\geq\bar\varepsilon_{k,l}(T')$.
\end{proof}

\begin{lemma}\label{lem4}
	Assume that  $T\ncong K_{1,n-1}$, $n\geq k+1$.  We have $$\bar\varepsilon_{k,k}(T)\geq\bar\varepsilon_{k,k}(T').$$
\end{lemma}
\begin{proof}
	We only need to check the cases that $\varepsilon_{k,k}(S;T)=d_T(S)$ may change, as follows:
	\begin{enumerate}
		\item $\bm{x}'=(1,0,1,0)$. It follows that $\varepsilon_{k,k}(S;T)-\varepsilon_{k,k}(S;T')=s$.
		\item $\bm{x}'=(0,1,1,0)$. It follows that $\varepsilon_{k,k}(S;T)-\varepsilon_{k,k}(S;T')=-s$.
		
		Recall that $\#V(Y)\geq \#V(X)$. We immediately obtain that the total sum of the above two cases is greater than $0$.
		\item $\bm{x}'=(1,0,1,1)$. We see that $\varepsilon_{k,k}(S;T)-\varepsilon_{k,k}(S;T')=d(S\cap V(P), v_s)$.
		\item $\bm{x}'=(0,1,1,1)$. We see that $\varepsilon_{k,k}(S;T)-\varepsilon_{k,k}(S;T')=-d(S\cap V(P), v_0)$.
		
		Recall that $\#V(Y)\geq \#V(X)$. By the symmetry of $P$, we similarly get that the total sum of the above two cases is greater than $0$.
		
		\item $\bm{x}'=(0,0,1,1)$. It follows that $\varepsilon_{k,k}(S;T)-\varepsilon_{k,k}(S;T')=d(S\cap V(P), v_0)-d(S\cap V(P), v_s)$. By the symmetry of $P$, the total sum of this case equals $0$.
	\end{enumerate}
	Combining all the cases together and then taking the arithmetic mean,  we derive that $\bar\varepsilon_{k,k}(T)\geq\bar\varepsilon_{k,k}(T')$.
\end{proof}
Combining Lemmas \ref{lem3} and \ref{lem4} yields
\begin{theorem}\label{thm1}
	Assume that  $T\ncong K_{1,n-1}$ and $n> k\geq l$.  We have $$\bar\varepsilon_{k,l}(T)\geq\bar\varepsilon_{k,l}(T').$$
\end{theorem}

From Lemma \ref{lem1} and Theorem \ref{thm1} we derive that
\begin{theorem}\label{thm2}
For $T'\ncong P_n$ and $n\geq k+1$, there exists a $\sigma^{-1}$-transformation such that $$\bar\varepsilon_{k,l}(T')\leq\bar\varepsilon_{k,l}(T).$$
\end{theorem}

We now give more restrictions on $\sigma$-transformations.

One can easily see that the number of leaves in $T$ satisfies $\#\ell(T)=\#\ell(T')$ if and only if $X\neq \{v_s\}$. We call the $\sigma$-transformation the leaf $\sigma$-transformation if $X\neq \{v_s\}$, and denote it by $\sigma_\ell$-transformation. Similarly, we denote the inverse of $\sigma$-transformation by $\sigma_\ell^{-1}$-transformation if $X$ is not a singleton graph. Thus, we are concerned with the Steiner $(k,l)$-eccentricity on trees under the condition of a given number of leaves. However, it is unnecessary  to illustrate the existence of $\sigma_\ell$-transformation, as it can be equivalently replaced by $(p,q)$-transformation, which is depicted in the following section.

One can select a diametrical path $P(u,v)$ such that $u$ is in $Y$ and $v$ is in $X\cup Z$. We call the $\sigma$-transformation under this restriction the diametrical $\sigma$-transformation and simply denote it by $\sigma_d$-transformation. For the $\sigma_d^{-1}$-transformation, we can select one endpoint of a diametrical path to be in $Y$ and the other endpoint of the diametrical path to be in $T'\backslash Y$. The existence of $\sigma_d$-transformation is plain, since we can let the pendant edge containing $v$  be $X\cup Z$. For the existence of $\sigma_d^{-1}$-transformation, let the branching vertex $s$ nearest $v$ be $v_0$, such that $d(s,v)\leq d(s,u)$. And then set $X\cup P=P(v,s)$. Note that $\#V(P)>0$. Hence we have $\#V(Y)\geq \#V(X)$, as desired.

One can check that the maximum degree of $T$, $\Delta(T)$, does not decrease if $X=\{v_s\}$ or is a path. We denote the $\sigma$-transformation by $\sigma_\delta$-transformation if  $X=\{v_s\}$ or is a path, and call it the degree $\sigma$-transformation. For the $\sigma_\delta^{-1}$-transformation, we also let $X=\{v_s\}$ or a path. By directly taking $X=\{v_s\}$, we can confirm the  existence of  $\sigma_\delta$-transformation and $\sigma_\delta^{-1}$-transformation.
\section{$(p,q)$-transformations for trees}\label{sec4}

The $(p,q)$-transformation was first investigated by Ili\'c \cite{ilic2012extremal}. Let $w$ be a vertex of tree $T$. For given positive integers $p$ and $q$, let $T(p,q)$ be the tree from $T$ by attaching to $w$ two pendant paths of lengths $p$ and $q$. We say that $T(p+1,q-1)$ comes from $T(p,q)$ by $(p,q)$-transformation; please refer to Figure \ref{Fig2}. 

\begin{figure}[h!]
	\includegraphics[width=1\textwidth]{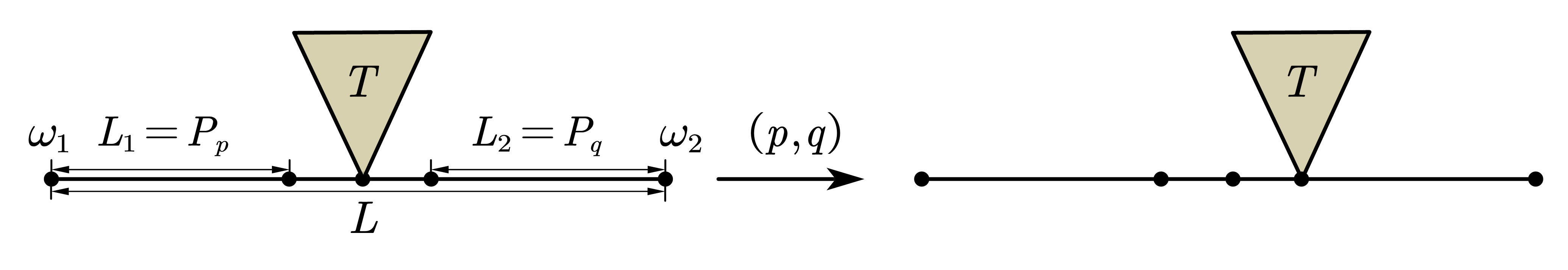}
	\caption{$(p,q)$-transformation}\label{Fig2}
\end{figure}

\begin{lemma}\label{lem5}
We have $\bar{\varepsilon}_{k,l}(T(p,q))> \bar{\varepsilon}_{k,l}(T(p+1,q-1))$, where $q\geq p+2$ and $n>k>l$.
\end{lemma}
\begin{proof}
Let $\alpha(S,T)=\max\{d(u,v): u\in S~ \text{and}~ \exists v\in \ell(T),~ \text{such that}~P(u,v)\cap S=\{u\} \}$. We consider the following cases.
\begin{enumerate}
	\item $\exists T_{k,l}(S)$ such that $EV(T_{k,l}(S))\cap L= \emptyset$.  
	\begin{enumerate}
		\item $S\cap L=\emptyset$.
		\item $S\cap L_1\neq \emptyset$, $S\cap L_2\neq \emptyset$.
		
		Obviously, the above two subcases do not change $|T_{k,l}(S)|$.
		
		\item $S\cap L_1\neq \emptyset$, $S\cap L_2= \emptyset$. We have $\alpha(S,T)\geq q$ and $|T'_{k,l}(S)|=|T_{k,l}(S)|+1$. Let $\mathcal S_c$ be the set of all $k$-sets which satisfy this subcase.
		
		\item  $S\cap L_1= \emptyset$, $S\cap L_2\neq \emptyset$. We have $\alpha(S,T)\geq\max\{p, d(S,\omega_2)\}$. Moreover, 
		\begin{equation*}
		|T'_{k,l}(S)|=	\begin{cases}
				|T_{k,l}(S)|, &\alpha(S,T)=p\geq d(S,\omega_2),\cr
					|T_{k,l}(S)|-1, &\text{else}.
			\end{cases}
		\end{equation*}
		Let $\mathcal S_d$ be the set of all $k$-sets which satisfy this subcase and make $|T'_{k,l}(S)|=|T_{k,l}(S)|-1$ hold.
	\end{enumerate}
	Recall that $q\geq p+2$. This means that set $\{S: \alpha(S,T)\geq q\}$ is a subset of $\{S: \alpha(S,T)\geq \max\{p,q-1\}\}\backslash \{S: \alpha(S,T)=p\geq d(S,\omega_2)\}$. Moreover, for fixed $S\cap T$, where $S\subseteq \{S: \alpha(S,T)\geq q\}$ and $\#(S\cap T)<k$,  the cardinality of $\{S: S\cap L_1\neq\emptyset \}$ is less than that of $\{S: S\cap L_2\neq\emptyset \}$. We immediately get $\#\mathcal{S}_c\leq\#\mathcal{S}_d$.
		
	Hence the total sum of this case satisfies
	\begin{equation*}
		\sum \varepsilon_{k,l}(S;T(p,q))\geq 	\sum \varepsilon_{k,l}(S;T(p+1,q-1)).
	\end{equation*}
	\item  $\exists T_{k,l}(S)$ such that $EV(T_{k,l}(S))\cap L=\{\omega_1,\omega_2\}$. Note that $d(\omega_1,S)\geq \alpha(S,T)$. If $S\cap L_2\neq \emptyset$, then  $|T_{k,l}(S)|$ does not change. If $S\cap L_2= \emptyset$, then $|T_{k,l}(S)|$ does not change due to $d(\omega_2,S)=q>p\geq d(\omega_1,S)$.
	\item $\forall T_{k,l}(S)$ such that $EV(T_{k,l}(S))\cap L= \{\omega_1\}$ or $\{\omega_2\}$. 
	\begin{enumerate}
		\item $\omega_1$ and $\omega_2$ can be both obtained. We thus have $d(\omega_1,S)=d(\omega_2, S)$ and $S\cap L_2\neq \emptyset$. Hence  $|T_{k,l}(S)|$ does not change.
		\item  $\forall T_{k,l}(S)$, such that $EV(T_{k,l}(S))\cap L= \{\omega_1\}$. It must have $S\cap L_2\neq \emptyset$ and hence $|T'_{k,l}(S)|=|T_{k,l}(S)|$.
		\item  $\forall T_{k,l}(S)$, such that $EV(T_{k,l}(S))\cap L= \{\omega_2\}$. We have
		$d(\omega_2, S)>d(\omega_1,S)$, $d(\omega_2,S)\geq \alpha(S,T)$, and  $\#\{u: d(u,S)=d(\omega_2,S), u\in\ell(T)\}\leq k-l-1$. If $S\cap L_1\neq\emptyset$, then $|T'_{k,l}(S)|=|T_{k,l}(S)|$; if $S\cap L_1=\emptyset$, then 	$|T'_{k,l}(S)|=|T_{k,l}(S)|-1$.
	\end{enumerate}
Hence the total sum of this case satisfies
	\begin{equation*}
		\sum \varepsilon_{k,l}(S;T(p,q))>	\sum \varepsilon_{k,l}(S;T(p+1,q-1)).
	\end{equation*}
\end{enumerate}
Considering  all the above situations and taking the arithmetic mean, we obtain  $\bar{\varepsilon}_{k,l}(T(p,q))> \bar{\varepsilon}_{k,l}(T(p+1,q-1))$ and the proof is complete.
\end{proof}
\begin{remark}
Note that this conclusion may not hold for general graphs since $\alpha(S,T)$ depends on Proposition \ref{prop1}, or more essentially, $\alpha(S,T)$ relies on the acyclicity of trees. 
\end{remark}
\begin{lemma}\label{lem6}
We have $\bar{\varepsilon}_{k,k}(T(p,q))> \bar{\varepsilon}_{k,k}(T(p+1,q-1))$, where $q\geq p+2$ and $n>k$.
\end{lemma}
\begin{proof}
One can easily see that  $|T'_{k,l}(S)|=|T_{k,l}(S)|$ if one of the following conditions holds:
\begin{enumerate}
	\item $T_{k,k}(S)\cap L_1=T_{k,k}(S)\cap L_2=\emptyset$;
	\item $T_{k,k}(S)\cap L_1\neq\emptyset$, $T_{k,k}(S)\cap L_2\neq\emptyset$.
\end{enumerate}

We now consider two cases:
\begin{enumerate}
	\item[(3)]  $T_{k,k}(S)\cap L_1\neq\emptyset$, $T_{k,k}(S)\cap L_2=\emptyset$. 	We have $|T'_{k,l}(S)|=|T_{k,l}(S)|+1$.
	\item[(4)]  $T_{k,k}(S)\cap L_1=\emptyset$, $T_{k,k}(S)\cap L_2\neq\emptyset$.  We have $|T'_{k,l}(S)|=|T_{k,l}(S)|-1$.
\end{enumerate}
For fixed $S\cap T\subsetneqq S$, the choices of vertices in $\{u: u\in L_1\cap S\}$ are strictly less than those in  $\{u: u\in L_2\cap S\}$ due to $q\geq p+2$. It means that the total sum of $\varepsilon_{k,k}(S)$ in the above two cases is strictly decreased. Therefore, we obtain $\bar{\varepsilon}_{k,k}(T(p,q))> \bar{\varepsilon}_{k,k}(T(p+1,q-1))$.
\end{proof}
\begin{theorem}\label{thm3}
We have 
\begin{enumerate}
	\item  $\bar{\varepsilon}_{k,l}(T(p,q))> \bar{\varepsilon}_{k,l}(T(p+1,q-1))$, if $q\geq p+2$;
	\item  $\bar{\varepsilon}_{k,l}(T(p,q))= \bar{\varepsilon}_{k,l}(T(p+1,q-1))$, if $q=p+1$;
	\item  $\bar{\varepsilon}_{k,l}(T(p,q))\leq \bar{\varepsilon}_{k,l}(T(p+1,q-1))$, if $q\leq p$.
\end{enumerate}
\end{theorem}
\begin{proof}
Case (1) comes from Lemmas \ref{lem5} and \ref{lem6}. Case (2) holds by using symmetry. Case (3) is a special case of $\sigma$-transformation.
\end{proof}
\section{Lower and upper bounds for general trees}\label{sec5}

\begin{figure}[h!]
	\includegraphics[width=1\textwidth]{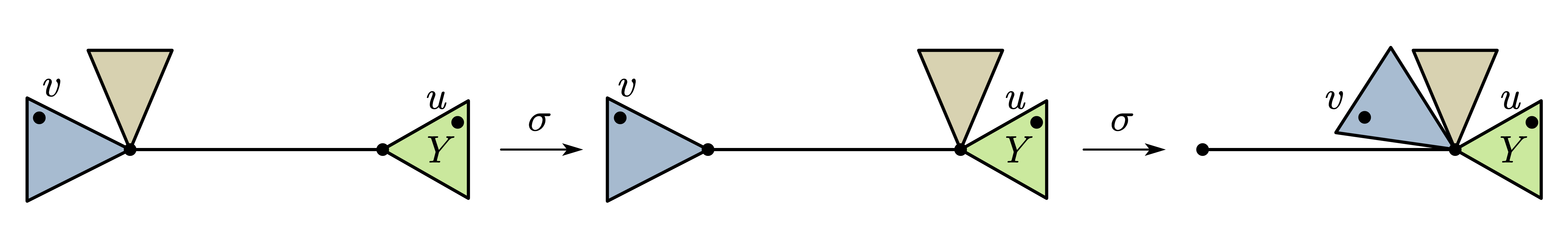}
	\caption{The number of diametrical paths decreases by one after at most two $\sigma_d$-transformations.}\label{Fig3}
\end{figure}
\begin{lemma}\label{lem7}
(1) The diameter of tree $T\ncong K_{1,n-1}$ strictly decrease after several $\sigma_d$-transformations. 

(2) The diameter of tree $T\ncong P_n$ strictly increases after several $\sigma_d^{-1}$-transformations.
\end{lemma}
\begin{proof}
	We claim that the number of diametrical paths decreases by exactly one after no more than two $\sigma_d$-transformations. If the intersection of $Z$ and the selected diametrical path $P(u,v)$ is nonempty, then the number of diametrical paths decreases by exactly one after $\sigma_d$-transformation. Otherwise, $P(u,v)\cap X\neq \emptyset$ and thus we need to use $\sigma_d$-transformation twice, see Figure \ref{Fig3}. Since the number of diametrical paths is finite, we can apply $\sigma_d$-transformations to reduce the diameter until it becomes two. When $\operatorname{diam}(T)=2$, i.e. $T\cong K_{1,n-1}$, there are no choices of $Z$ in the $\sigma$-form of $T$. We thus establish the first conclusion.
	
	Since $T\ncong P_n$, it is clear that a diametrical path $P(u,v)$ passes through a branching vertex of $T$. We let this branching vertex be $v_0$ and the branch at $v_0$ with the maximum depth be $Y$. Now the endpoint $v$ of $P(u,v)$ is in $X\cup Z$. If $v\in Z$ and $v\notin X$, then the diameter of $T$ strictly increases after $\sigma_d$-transformation. If $v\in X$, we can move $Z$ along the path $P(v,v_s)$ to make the diameter of $T$ strictly increase by applying $\sigma_d^{-1}$-transformation. This procedure can proceed until $T\cong P_n$. When $T\cong P_n$, we have $\Delta(T)=2$, and thus the selection of $Z$ in the $\sigma$-form of $T$ can not be achieved.
\end{proof}
\begin{theorem}
Suppose that $T$ is a tree of order $n\geq k+1$. We have $$\bar{\varepsilon}_{k,l}(K_{1,n-1})\leq \bar{\varepsilon}_{k,l}(T)\leq \bar{\varepsilon}_{k,l}(P_n),$$ and the bounds are sharp.
\end{theorem}
\begin{proof}
By Lemma \ref{lem7}, we apply $\sigma_d$-transformation on $T$ again and again until its diameter becomes $2$. Hence, $T$ is transformed to be a star. From Theorem \ref{thm1}, we have $\bar{\varepsilon}_{k,l}(K_{1,n-1})\leq \bar{\varepsilon}_{k,l}(T)$.

Similarly, using $\sigma_d^{-1}$-transformation on $T$ several times, $T$ switches into a path. From Theorem \ref{thm2}, we have $\bar{\varepsilon}_{k,l}(T)\leq \bar{\varepsilon}_{k,l}(P_n)$.
\end{proof}

Since the number of leaves keeps constant after $\sigma_\ell$-transformation, it is attractive to consider a tree with a fixed number of leaves. 

A starlike tree is a tree with exactly one branching vertex, and we denote the degree of this branching vertex as $p$. Note that the maximum  degree of starlike tree is also $p$. Concisely, the starlike tree $S(n_1,n_2,\ldots,n_p)$ has a branching vertex $v$ such that 
\begin{equation*}
	S(n_1,n_2,\ldots,n_p)-v=\bigcup\limits_{i=1}^pP_{n_i},
\end{equation*}
where $n_1\geq n_2\geq \cdots\geq n_p$. 

For the pendant paths of the starlike tree, if they pairwise differ by at most one, then we call the starlike tree balanced and denote it by $BS(n,p)$. More generally, a balanced starlike tree is uniquely determined by its order $n$ and the maximum degree $p$.
See Figure \ref{Fig4} as an example.

\begin{figure}[h!]
	\includegraphics[width=0.25\textwidth]{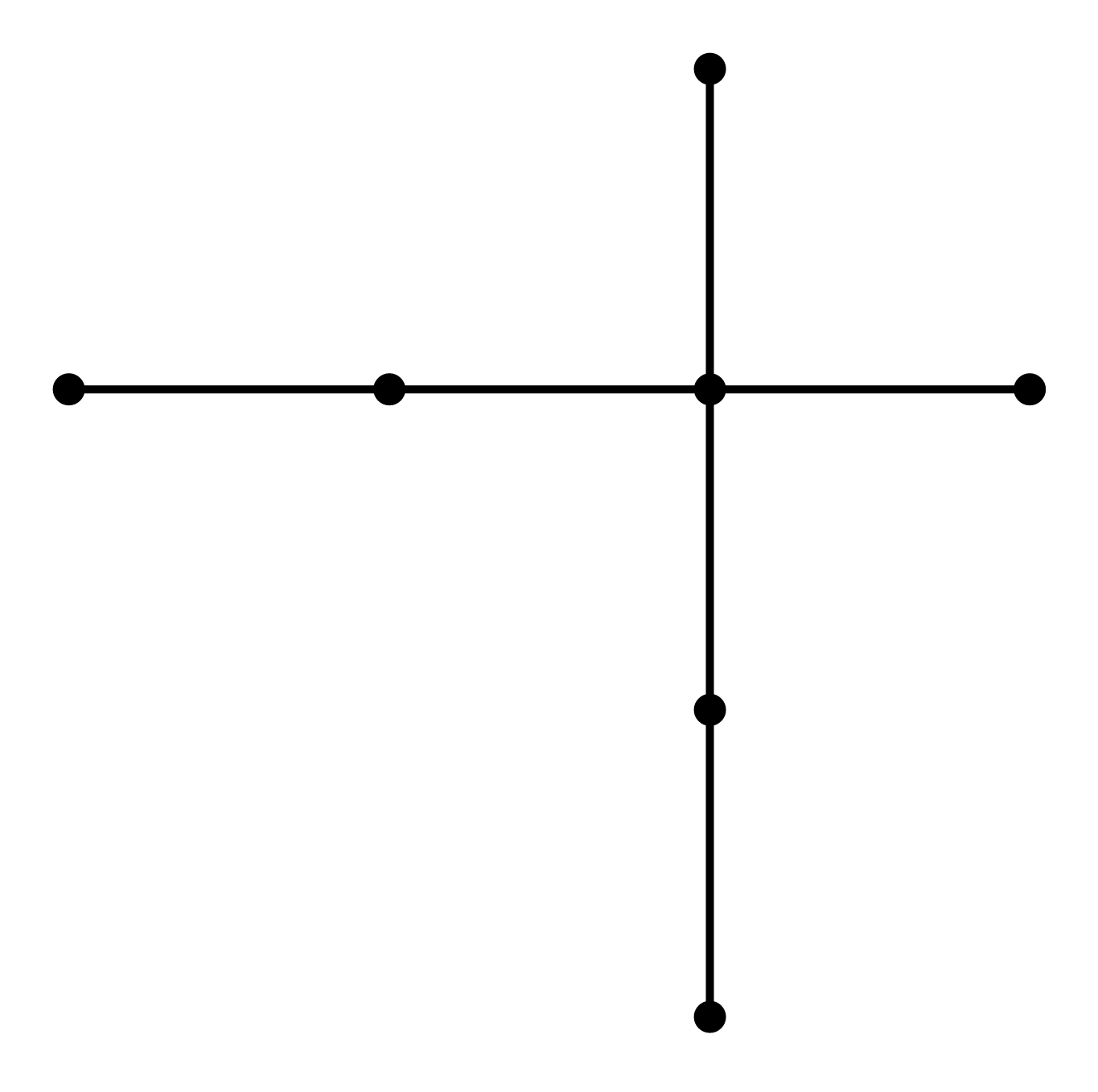}
	\caption{$BS(7,4)$}\label{Fig4}
\end{figure}

\begin{theorem}\label{thm5}
If $T$ is tree with $n$ vertices and $p$ leaves, then $$\bar{\varepsilon}_{k,l}(T)\geq \bar{\varepsilon}_{k,l}(BS(n,p)).$$
\end{theorem}
\begin{proof}
We consider $p\geq 3$, and thus tree $T$ has at least one branching vertex.

\begin{enumerate}
	\item  Assume that tree $T$ has no less than $2$ branching vertices. Our aim is to decrease the number of branching vertices to exactly one.  We choose two nearest branching vertices $x$ and $y$. Clearly, $x$ and $y$ are connected by a path $P$. Without loss of generality, let $y$ be the root owning the maximum depth of subtree in $T\backslash P$ and then let $x$ be the endpoint of a pendant path with a reasonably small length. This can be achieved because if it cannot be done, then the branching vertices on those branches of $x$, which contain pendant paths, can be selected as the new $x$. We now let one of the pendant paths of $x$ be $X$ in the $\sigma$-form of $T$. The choices of $Y$ and $Z$ are now obvious. 
	
	After $\sigma_\ell$-transformation, the degree of $x$ becomes two, i.e., the number of branching vertices decrease by one. Since $T$ owns finite branching vertices, we can continue this procedure until $T$ has exactly one branching vertex.
	\item  Assume that tree $T$ has exactly one branching vertex and is unbalanced. Thus $T\cong S(n_1,n_2,\ldots, n_p)$ is an unbalanced starlike tree. Our aim is to switch the unbalanced starlike tree into a balanced starlike tree. Let $i_1$ and $i_2$ be the minimum and maximum value of index $i$ in sequence $\{n_i\}$ in turn, such that $n_{i_1}-n_{i_2}\geq 2$. We now apply $(p,q)$-transformation on $T$ to transform $T$ into $S(n_1,\ldots, n_{i_1}-1,\ldots, n_{i_2}+1,\ldots,n_p)$. This operation can be repeatedly employed until $T$ becomes balanced.
\end{enumerate}
According to Theorems \ref{thm1} and \ref{thm3}, all the operations to transform $T$ into a balanced starlike tree do not decrease the average Steiner $(k,l)$-eccentricity. Therefore, we conclude that $\bar{\varepsilon}_{k,l}(T)\geq \bar{\varepsilon}_{k,l}(BS(n,p))$.
\end{proof}

We next present the criteria for checking the average Steiner $(k,l)$-eccentricity of starlike trees. Let $x=(x_1,x_2,\ldots, x_p)$ and $y=(y_1,y_2,\ldots,y_p)$ be two $p$-tuples of positive integers. We write $y\prec x$ if $x$ and $y$ satisfy the following conditions:
\begin{enumerate}
	\item $x_1\geq x_2\geq\cdots\geq x_p$ and  $y_1\geq y_2\geq\cdots\geq y_p$;
	\item $\sum_{i=1}^kx_i\geq \sum_{i=1}^ky_i$, for every $1\leq k<p$;
	\item $\sum_{i=1}^px_i= \sum_{i=1}^py_i$.
\end{enumerate}
\begin{theorem}\label{thm6}
	Let $x=(x_1,x_2,\ldots, x_p)$ and $y=(y_1,y_2,\ldots,y_p)$ be two $p$-tuples with $p\geq 2$, such that $y\prec x$ and $n-1=\sum_{i=1}^px_i= \sum_{i=1}^py_i$. Then 
	\begin{equation}\label{eq:thm8}
		\bar{\varepsilon}_{k,l}(S(x))\geq \bar{\varepsilon}_{k,l}(S(y)),
	\end{equation}
	with equality if and only if $S(x)\cong S(y)$.
\end{theorem}
\begin{proof}
	It is easy to see that $(p,q)$-transformation can be applied for starlike trees. We will prove the theorem by induction on $p$. For $p=2$, it is clear that  the inequality \eqref{eq:thm8} holds. Assume that the inequality \eqref{eq:thm8} holds for all $p<k$. For $p=k$, we need to consider two cases.
	\begin{enumerate}
		\item If there exists $1\leq m<k$ such that $x_1+x_2+\cdots+x_m=y_1+y_2+\cdots+y_m$, we can divide $(S(y_1,y_2,\ldots, y_p))$ into two parts $S(y_1,\ldots,y_m)\cup S(y_{m+1},\ldots,y_k)$. Applying the induction hypothesis, we transform $S(y_1,\ldots,y_m)\cup S(y_{m+1},\ldots,y_k)$ into 
		$S(x_1,\ldots,x_m)\cup S(x_{m+1},\ldots,x_k)=S(x_1,x_2,\ldots, x_p)$.
		\item Otherwise, $x_1+x_2+\cdots+x_m>y_1+y_2+\cdots+y_m$ for all $1\leq m<k$. Notice that now $y_k>x_k\geq 1$. We can switch $S(y_1,y_2,\cdots,y_k)$ into $S(y_1+1,y_2,\cdots,y_{k-1},y_k-1)$. The condition $x\succ y$ is still preserved. 
	\end{enumerate} 
	We can use the above two cases recurrently until $y$ transforms into  $x$. Note that each step does not decrease the average Steiner $(k,l)$-eccentricity. 
\end{proof}

A broom $B(n,\Delta)$ is a tree of order $n$ and maximum degree $\Delta$, constructed by attaching one endpoint of $P_{n-\Delta}$ to an arbitrary leaf of $K_{1,\Delta}$. A broom is also a special starlike graph with only one pendant path of length more than $1$. Figure \ref{Fig5} illustrates a broom with maximum degree $6$. Using $\sigma$-transformation and Theorem \ref{thm1}, one can easily verify that 
\begin{equation*}
	\bar{\varepsilon}_{k,l}(P_n)\geq \bar{\varepsilon}_{k,l}(B(n,3))\geq \bar{\varepsilon}_{k,l}(B(n,4))\geq \cdots \geq \bar{\varepsilon}_{k,l}(B(n,n-1))=\bar{\varepsilon}_{k,l}(K_{1,n-1}).
\end{equation*}
It follows that $B(n, \max\{3, k-l+1\})$ is the second maximum average Steiner $(k,l)$-eccentricity among trees of order $n$. 

\begin{figure}[h!]
	\includegraphics[width=0.35\textwidth]{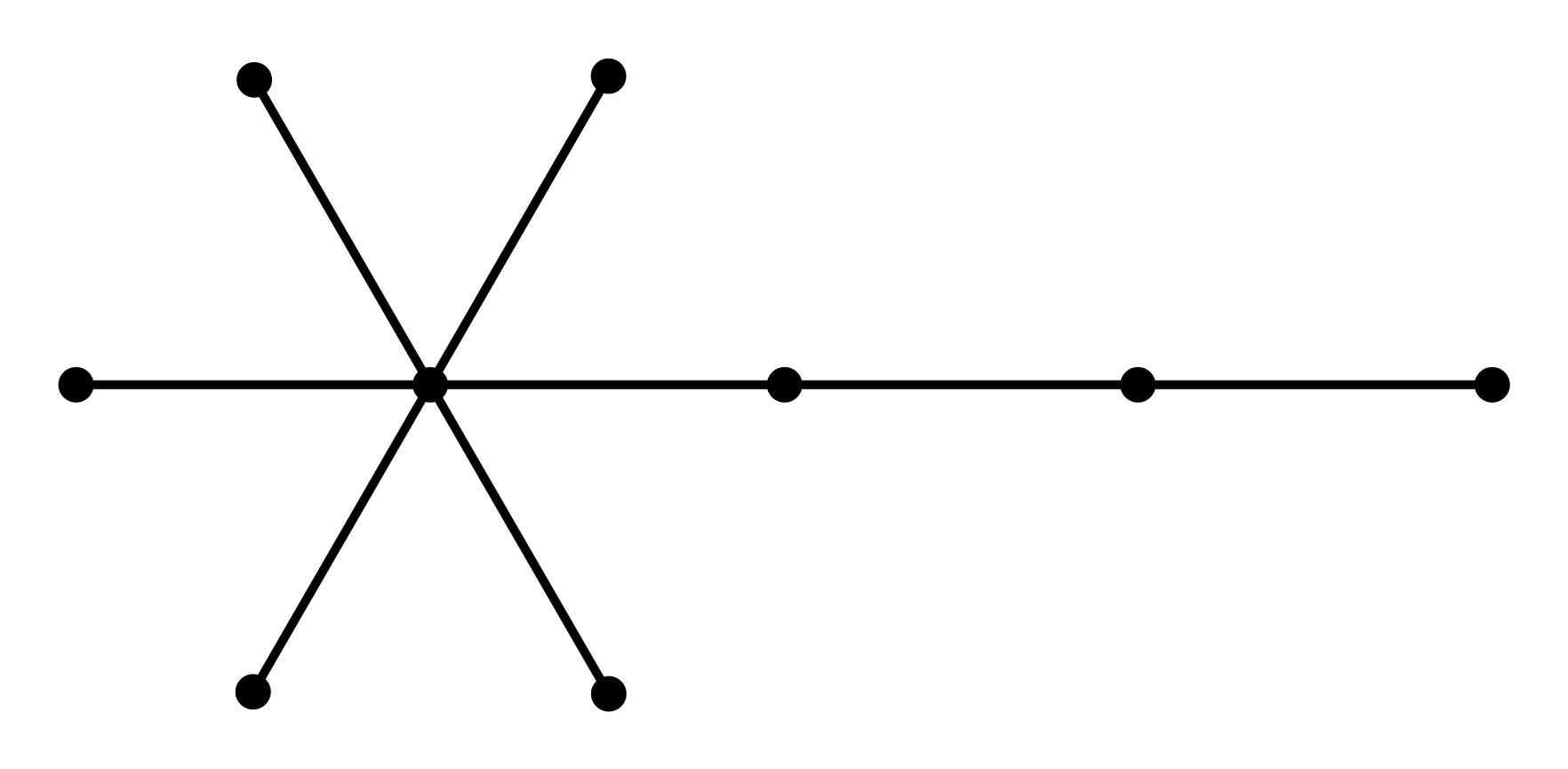}
	\caption{$B(9,6)$}\label{Fig5}
\end{figure}

\begin{corollary}\label{cor1}
	For every starlike tree $T$ of order $n$ and with $p$ leaves it holds
	\begin{equation*}
		\bar{\varepsilon}_{k,l}(T)\leq \bar{\varepsilon}_{k,l}(B(n,p)).
	\end{equation*}
\end{corollary}
\begin{proof}
	Notice that $B(n,p)\cong S(n-p,1,\ldots,1)$ and for any $T\cong S(y_1,y_2,\ldots,y_p)$, it has $(n-p,1,\ldots,1)\succ (y_1,y_2,\ldots,y_p)$. By Theorem \ref{thm6}, $B(n,p)$ attains the maximum value of $\bar{\varepsilon}_{k,l}(T)$.
\end{proof}
\begin{theorem}\label{thm7}
	If $T$ is tree of order $n$ and maximum degree $\Delta$, then $$\bar{\varepsilon}_{k,l}(T)\leq \bar{\varepsilon}_{k,l}(B(n,\Delta)).$$
\end{theorem}
\begin{proof}
Let $\delta$ be a vertex in $T$ with maximum degree $\Delta$. One can partition $T$ into $\Delta$ maximum subtrees of $T$ with root vertex $\delta$.  We denote them by $ST_1,ST_2,\ldots, ST_\Delta$ in turn. We now prove that every $ST_i$ can be changed into a path. We can use $(p,q)$-transformation repeatedly at any branching vertex with largest distance from the root node $\delta$ until $ST_i$ becomes a path. Here we can let $p\geq q$ in the $(p,q)$-transformation.  When all $ST_i$ turn into paths, $T$ arrives at the starlike tree and the average Steiner $(k,l)$-eccentricity does not decrease.  We now use Theorem \ref{thm3} and Corollary \ref{cor1} to get $\bar{\varepsilon}_{k,l}(T)\leq \bar{\varepsilon}_{k,l}(B(n,\Delta))$.
\end{proof}

\begin{theorem}\label{thm8}
	If $T$ is tree with $n$ vertices and $p$ leaves, then $$\bar{\varepsilon}_{k,l}(T)\leq \bar{\varepsilon}_{k,l}(B(n,p)),$$
	where $k>l$.
\end{theorem}
\begin{proof}
	For nontrivial cases, we assume that $p\geq 3$, so $T$ now has at least one branching vertex. The inequality $k>l$ implies the restriction $\#V(X)\leq\#V(Y)$ of $\sigma$-transformation ($\sigma_\ell$-transformation) and its inverse can be canceled, which is evidenced by the proof of Lemma \ref{lem3}.  We thus have the following discussion.
	\begin{enumerate}
		\item $T$ has at least $2$ branching vertices. We choose two nearest branching vertices as $v_0$ and $v_s$ in the $\sigma^{-1}$-form of $T$. By applying $\sigma^{-1}_\ell$-transformation once, the degree of $v_0$ can become exactly $2$. Since the branching vertices of tree $T$ are finite, we can use $\sigma^{-1}_\ell$-transformation again and again to obtain a starlike tree.
		\item Now suppose that $T$ has exactly one branching vertex, or equivalently, it is a starlike tree. We  apply $(p,q)$-transformation repeatedly to achieve the broom $B(n,p)$, as depicted in the proof of Corollary \ref{cor1}.
	\end{enumerate}
	All transformations used to reach $B(n,p)$ keep the number of leaves and do not decrease the average Steiner $(k,l)$-eccentricity, so we derive that $\bar{\varepsilon}_{k,l}(T)\leq \bar{\varepsilon}_{k,l}(B(n,p))$.
\end{proof}

A caterpillar is a tree with the property that only leaves remain if we delete a diametrical path. We call the diametrical path of a caterpillar the spine.  Let $P_{d+1}=u_0u_1\cdots u_d$ be a diametrical path of $T$. Notation $CP_n(p_1,p_2,\ldots, p_{d-1})$, where $p_i\geq0$, represents a caterpillar of order $n$ where each $u_i$ in $P_{d+1}$ owns $p_i$ extra pendant edges for every $i\in\{1,2,\ldots,d-1\}$. 

Notice that a broom is a special caterpillar $CP_n(\Delta-2,0,\ldots,0)$. For $d=2$, the caterpillar is uniquely determined by $K_{1,n-1}$. For $d\geq 3$, we now introduce two kinds of caterpillars: 
\begin{enumerate}
	\item Central type. All leaves are joined with central vertices. In detail, for $d$ even, we attach $n-d-1$ vertices to $u_{\frac{d}{2}}$ and denote such a caterpillar by $T_{n,d}^C$ (or $T_{n,d}^C(n-1-d)$); for $d$ odd, we attach $s$ vertices to $u_{\lfloor\frac{d}{2}\rfloor}$ and $n-d-1-s$ vertices to $u_{\lceil\frac{d}{2}\rceil}$, and we denote it by $T_{n,d}^C(s)$.
	\item Double comet type. $s$ new vertices are joined with $u_1$ and the remaining $n-d-1-s$ vertices are attached to $u_{d-1}$, and we denote it by $D(n, s,n-d-1-s)$. 
\end{enumerate} 

\begin{figure}[h!]
	\includegraphics[width=1\textwidth]{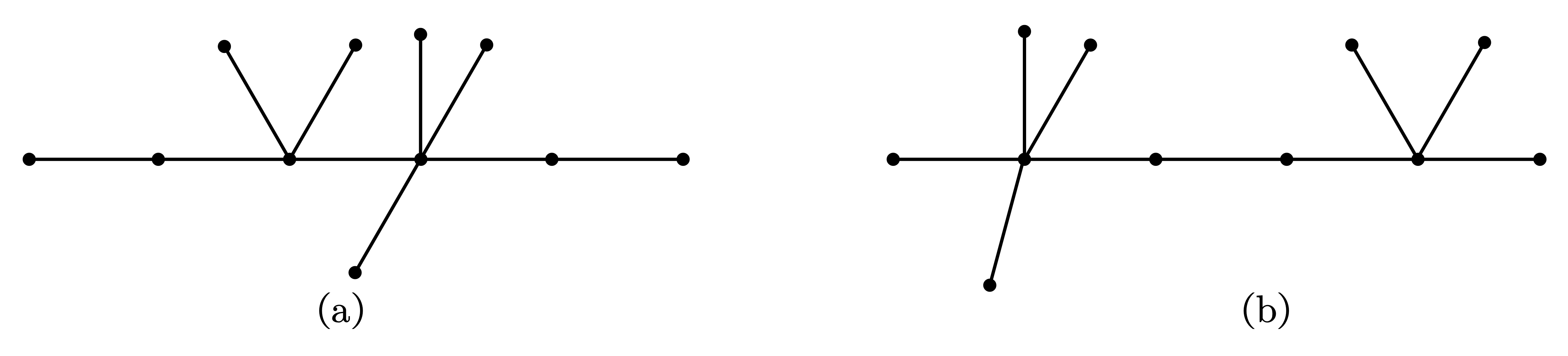}
	\caption{(a) $T_{11,5}^C(2)$; (b) $D(11,3,2)$}\label{Fig6}
\end{figure}

See Figure \ref{Fig6} for examples. In fact, the parameter $s$ can be ignored due to the following lemma. 
\begin{lemma}
Assume that the diameter $d$ and the order $n$ of the caterpillar are fixed. For $d$ odd, $\bar{\varepsilon}_{k,l}(T_{n,d}^C(s))$ keeps a  constant for every $s\in\{0,1,\ldots,n-d-1\}$.  
$\bar{\varepsilon}_{k,l}(D(n, s,n-d-1-s))$  keeps a  constant for every $s\in\{0,1,\ldots,n-d-1\}$.
\end{lemma}
\begin{proof}
$\sigma$-transformation can be applied for any pairs of different $T_{n,d}^C(s)$. Note that the equality condition holds for such a $\sigma$-transformation. 

Due to $\varepsilon_X(v_s)=\varepsilon_Y(v_0)=1$ in $\sigma^{-1}$-form of $T'$, the $\sigma^{-1}$-transformation  for any distinct  $D(n, s,n-d-1-s)$ satisfies the equality condition.
\end{proof}
Therefore, we simply denote the central type caterpillar and the double comet type caterpillar by $T_{n,d}^C$ and $T_{n,d}^D$, respectively. We also have $T_{n,2}^C=T_{n,2}^D=K_{1,n-1}$.

\begin{theorem}\label{thm9}
	If $T$ is tree of order $n$ and $\operatorname{diam}(T)=d$, then $$\bar{\varepsilon}_{k,l}(T)\geq \bar{\varepsilon}_{k,l}(T_{n,d}^C).$$
\end{theorem}
\begin{proof}
	Let $T$ be a tree with diameter $d$. We can select a diametrical path $P_{d+1}=u_0u_1\cdots u_d$ of $T$. Let $T_i$ be the maximum subtree taking $u_i$ as the root and not containing  $P_{d+1}$ for each $i\in\{1,2,\ldots,d-1\}$. Condition $\varepsilon_2(u_i;T_i)\leq \min\{i,d-i\}$ ensures that we can use $\sigma_\delta$-transformation on $T_i$, as long as $T_i$ does not become a star with central vertex $u_i$. Applying this process for all $T_i$, $T$ turns into a caterpillar. We now use $\sigma_d$-transformation again and again to move all vertices in $T\backslash P_{d+1}$ to the central vertices of $P_{d+1}$. Hence a central type caterpillar $T_{n,d}^C$ can be obtained. Due to the non-increasing properties of the above transformations, we finally see that $\bar{\varepsilon}_{k,l}(T)\geq \bar{\varepsilon}_{k,l}(T_{n,d}^C)$.
\end{proof}
\begin{corollary}
	If $T$ is a tree of order $n$ and $\operatorname{diam}(T)=d$, then $T_{n,d}^C(n-d-1)$ have both maximum degree and minimum average Steiner $(k,l)$-eccentricity.
\end{corollary}
From the relationship between $\operatorname{diam}(T)$ and $\operatorname{rad}(T)$, we straightly have 
\begin{corollary}
If $T$ is a tree with $n$ vertices and $\operatorname{rad}(T)=r$, then
$$\bar{\varepsilon}_{k,l}(T)\geq \bar{\varepsilon}_{k,l}(T_{n,2r-1}^C).$$
\end{corollary}
Let $x=(x_1,x_2,\ldots, x_{d-1})$ and $y=(y_1,y_2,\ldots, y_{d-1})$ be two $(d-1)$-tuples of non-negative integers, where $d\geq 2$. We write $x\vartriangleright y$ if $x$ and $y$ satisfy the following conditions:
\begin{enumerate}
	%\item $x_1\geq x_2\geq\cdots\geq x_{\lfloor\frac{d}{2}\rfloor}$, $x_{d-1}\geq x_{d-2}\geq \cdots\geq x_{\lceil\frac{d}{2}\rceil}$,  $y_1\geq y_2\geq\cdots\geq y_{\lfloor\frac{d}{2}\rfloor}$ and $y_{d-1}\geq y_{d-2}\geq \cdots\geq y_{\lceil\frac{d}{2}\rceil}$;
	\item $\sum_{i=1}^kx_i\geq \sum_{i=1}^ky_i$ and  $\sum_{i=1}^kx_{d-i}\geq \sum_{i=1}^ky_{d-i}$, for every $1\leq k<\lfloor\frac{d}{2}\rfloor$;
	\item $\sum_{i=1}^{d-1}x_i=\sum_{i=1}^{d-1}y_i$.
\end{enumerate} 
\begin{theorem}
Let $x$ and $y$ be two $(d-1)$-tuples of non-negative integers such that $x\vartriangleright y$ and $n-d-1=\sum_{i=1}^{d-1}x_i=\sum_{i=1}^{d-1}y_i$. Then
\begin{equation*}
	\bar{\varepsilon}_{k,l}(CP_n(x))\geq \bar{\varepsilon}_{k,l}(CP_n(y)).
\end{equation*}
\end{theorem}
\begin{proof}
For $CP_n(x)$, one can keep $y_i$ (or $y_{d-i}$) leaves attaching to $u_{i}$ (or $u_{d-i}$) and move extra leaves to $u_{i+1}$ (or $u_{d-i-1}$) by $\sigma$-transformation from $i=1$ to $\lfloor\frac{d}{2}\rfloor$ in turn.

Next, only vertices $u_{\lfloor\frac{d}{2}\rfloor}$ and $u_{\lceil\frac{d}{2}\rceil}$ need to be considered, where $d$ is odd. Recall the equality condition of $\sigma$-transformation. It follows that movements of leaves between $u_{\lfloor\frac{d}{2}\rfloor}$ and $u_{\lceil\frac{d}{2}\rceil}$ do not change the average Steiner $(k,l)$-eccentricity of the caterpillar. So we can switch $CP_n(x)$ to $CP_n(y)$ by using $\sigma$-transformation. By Theorem \ref{thm1}, $\bar{\varepsilon}_{k,l}(CP_n(x))\geq \bar{\varepsilon}_{k,l}(CP_n(y))$.
\end{proof}
\begin{corollary}\label{cor3}
	If $T$ is a caterpillar of order $n$ and $\operatorname{diam}(T)=d$, then 
	\begin{equation}\label{eq5.2}
		\bar{\varepsilon}_{k,l}(T_{n,d}^C)\leq\bar{\varepsilon}_{k,l}(T)\leq \bar{\varepsilon}_{k,l}(T_{n,d}^D).
	\end{equation}
\end{corollary}
\begin{proof}
	The left part of \eqref{eq5.2} is the corollary of Theorem \ref{thm9}. Assume that $$T\cong CP_n(y_1,y_2,\ldots, y_{d-1}).$$ We can get the right part of \eqref{eq5.2}, since
	\begin{equation*}
		(y_1,y_2,\ldots, y_{d-1})\triangleleft\begin{cases}(\sum_{i=1}^{d/2}y_i,0,\ldots,0, \sum_{i=1}^{d/2-1}y_{d-i}),& d~\text{is even},\cr
		(\sum_{i=1}^{\lfloor d/2\rfloor}y_i,0,\ldots,0, \sum_{i=1}^{\lfloor d/2\rfloor}y_{d-i}),& d~\text{is odd}.
		\end{cases}
	\end{equation*}
	
\end{proof}
According to Theorem \ref{thm7} and Corollary \ref{cor3}, we have 
\begin{corollary}
If $T$ is a caterpillar of order $n$ and $\operatorname{diam}(T)=d$, then $B(n,n-d-1)$ is  the caterpillar with both maximum degree and maximum average Steiner $(k,l)$-eccentricity.
\end{corollary}

Since $SW_k(G)=\binom{n}{k}\bar\varepsilon_{k,k}(G)$, some of our conclusions for the average Steiner $(k,l)$-eccentricity also hold for the Steiner $k$-Wiener index.
\begin{theorem}
	Assume that $T$ is a tree of order $n$. We have some inequalities about  the average Steiner $(k,l)$-eccentricity of $T$:
	\begin{enumerate}
		\item $SW_k(K_{1,n-1})\leq SW_k(T)\leq SW_k(P_n)$;
		\item If the maximum degree of $T$ is $\Delta$, then $SW_k(T)\leq SW_k(B(n,\Delta))$;
		\item If $\#\ell(T)=p$, then $SW_k(BS(n,p))\leq SW_k(T)$;
		\item If $\operatorname{diam}(T)=d$, then $SW_k(T_{n,d}^c)\leq SW_k(T)$.
	\end{enumerate}
\end{theorem}
\begin{proof}
(1) and (4) are the conclusions of \cite[Theorem 3.3]{li2016steiner} and \cite[Theorem 3.1]{lu2018sharp}, respectively, and can be reproved by the methods in this paper. (2) and (3) are corollaries of Theorems \ref{thm5}, \ref{thm7}, \ref{thm9} .
\end{proof}

\section*{Acknowledgments}
The research is partially supported by the Natural Science Foundation of China (No. 12301107) and the Natural Science Foundation of Shandong Province, China (No. ZR202209010046).
\bibliographystyle{unsrt}
\bibliography{ref}
\end{document}